\documentclass[11pt]{amsart}
\usepackage[english]{babel}
\usepackage{amssymb,amsmath,amsthm}
\usepackage[latin2]{inputenc}
\usepackage{amsfonts, verbatim, amscd}
\usepackage[foot]{amsaddr}

\reversemarginpar
\newtheorem{theorem}{Theorem}
\newtheorem{proposition}[theorem]{Proposition}
\newtheorem{lemma}[theorem]{Lemma}

\def\acknowledgment{\par\addvspace{17pt}\small\rmfamily
\trivlist\if!\ackname!\item[]\else
\item[\hskip\labelsep
{\bfseries\ackname}]\fi}

\def\C{\mathbb{C}}
\def\R{\mathbb{R}}

\newcommand{\du}{\mathrm{d}}

\begin{document}
\title{The non-linear Cousin problem for $J$-holomorphic maps}

\author{Uro\v s Kuzman}

\address[Uro\v s Kuzman]{Faculty of Mathematics and Physics, University of Ljubljana, 
and Institute of Mathematics, Physics and Mechanics, Jadranska 19, 
1000 Ljubljana, Slovenia, uros.kuzman@fmf.uni-lj.si}

%

\begin{abstract}
We solve a non-linear Cousin problem for $J$-holomorphic maps. That is, we provide a gluing method for the pseudoholomorphic maps defined on a Cartan pair of domains in $\C$.  
\end{abstract}
\maketitle

\section{Introduction}

Let $\Omega_1, \Omega_2\subset \C$ be bounded simply connected domains with $\mathcal{C}^{1}$-boundary. We say that $(\Omega_1,\Omega_2)$ is a \emph{good pair} if $\Omega_1\cap \Omega_2$ is simply connected with $\mathcal{C}^{1}$-boundary and  $$\overline{\Omega_1\setminus\Omega_2}\cap \overline{\Omega_2\setminus\Omega_1}=\emptyset.$$ 
That is, the sets $\Omega_1$ and $\Omega_2$ are 'well-glued' together in the sense that there exists a smooth cut-off function $\chi\colon \C\to[0,1]$ that equals $1$ on some neighborhood of $\overline{\Omega_1\setminus\Omega_2}$ and vanishes on some neighborhood of $\overline{\Omega_2\setminus\Omega_1}$. 

On such a pair of domains, the classical additive Cousin problem for holomorphic functions can be solved with estimates. In particular, for a bounded domain $\Omega\subset \C$ let $T_\Omega$ be the standard Cauchy-Green operator:
\begin{eqnarray}\label{CG}T_\Omega(f)(\zeta)=\frac{1}{\pi}\iint_{\Omega}\frac{f(z)}{\zeta-z}\, \du x \, \du y. \end{eqnarray}
Then given holomorphic maps $f_1\colon\Omega_1\to\C^n$ and $f_2\colon\Omega_2\to\C^n$ whose images are $\mathcal{C}^{0}$-close on the intersection $\Omega_1\cap\Omega_2$ the map   
\begin{eqnarray}\label{osnovni}\widehat{f}=\chi f_1+ (1-\chi)f_2 - T_{\Omega_1\cup\Omega_2}\left[\chi_{\bar{\zeta}}(f_1-f_2)\right]\end{eqnarray}  
is holomorphic on $\Omega_1\cup\Omega_2$ and $\mathcal{C}^{0}$-close to $f_j$ on $\Omega_j$, $j=1,2$. Indeed, the first property follows from the fact that $\left[  T_{\Omega_1\cup\Omega_2} f\right]_{\bar{\zeta}}=f$ on $\Omega_1\cup\Omega_2$ while the second is implied by the bounded operator norm of $T_\Omega\colon \mathcal{C}(\Omega)\to \mathcal{C}(\Omega)$. Therefore, we can say that we have 'glued' $f_1$ and $f_2$ into a holomorphic map $\widehat{f}$. In this paper we provide analogous constructions for maps that are holomorphic with respect to a non-integrable almost complex structure.

Let $J$ be a $2n\times 2n$ smooth matrix function defined on $\R^{2n}$ and satisfying $J^2 = -id$. A map $f\colon \Omega\to \R^{2n}$ of Sobolev class $f\in W^{1,p}(\Omega)$, $p>2$, is called $J$-holomorphic if its differential satisfies the equation $$df \circ J_{st} = J(f) \circ df.$$ Here $J_{st}$ corresponds to the multiplication by the imaginary unit in $\R^{2n}$. When $\det(J+J_{st})\neq 0$ this condition can be turned into an equivalent non-linear Cauchy-Riemann system
\begin{eqnarray}\label{A}\bar{\partial}_Jf= f_{\bar{\zeta}}+A(f)\overline{f_{\zeta}}=0,\end{eqnarray}
where $A(z)(v)=(J(z)+J_{st})^{-1}(J(z)-J_{st})(\bar{v})$ is a $n\times n$ complex matrix function \cite{ST3}. Therefore, we denote by $\mathcal{J}(U)$ the set of smooth structures for which $\det(J+J_{st})\neq 0$ on $U\subset\mathbb{R}^{2n}$ and by $\mathcal{J}=\mathcal{J}(\mathbb{R}^{2n})$. Further, let $\mathcal{O}_J(\Omega)\subset W^{1,p}(\Omega)$ be the set of $J$-holomorphic maps. The following theorem provides a solution of the local Cousin problem.

\begin{theorem}\label{ena}
Let $(\Omega_1,\Omega_2)$ be a good pair and let $J\in\mathcal{J}$. Let $\mathcal{W}_1\subset \mathcal{O}_J(\Omega_1)$ and $\mathcal{W}_2\subset \mathcal{O}_J(\Omega_2)$ be two families of maps that are precompact in the spaces $W^{1,p}(\Omega_1)$ and $W^{1,p}(\Omega_2)$, respectively. For every $\epsilon>0$ there is $\delta>0$ such that for every $f_1\in \mathcal{W}_1$ and $f_2\in \mathcal{W}_2$ satisfying $\left\|f_1-f_2\right\|_{W^{1,p}(\Omega_1\cap\Omega_2)}<\delta$ there exists a map $\widehat{f}\in \mathcal{O}_J(\Omega_1\cup\Omega_2)$ such that: 
$$\left\|\widehat{f}-f_1\right\|_{W^{1,p}(\Omega_1)}<\epsilon\;\;\textrm{ and }\;\;\left\|\widehat{f}-f_2\right\|_{W^{1,p}(\Omega_2)}<\epsilon.$$ 
\end{theorem}
\noindent A particular case of this theorem was proved in \cite[Proposition 6]{BERKUZ}. However, the present version is stronger since we remove the tedious assumption on the regularity of the structure $J$. Note that, in comparison to the classical problem (\ref{osnovni}), the derivatives of $f_1$ and $f_2$ have to obey certain $L^p$-bounds. This is due to the fact that $\bar\partial_Jf$ is non-linear and includes $f_\zeta$ as well. Hence the $L^{p}$-norm of $df$ has to be bounded if one wants to apply the implicit fuction theorem (see Theorem \ref{sest}). Moreover, we need precompactness of $\mathcal{W}_1$ and $\mathcal{W}_2$ in order to produce an uniformly bounded right inverse for the linearization of $\bar{\partial}_J$. Finally, note that $f_1$ and $f_2$ have to be $W^{1,p}$-close (and not just $\mathcal{C}^0$) on $\Omega_1\cap\Omega_2$. Nevertheless, since $p>2$, the Sobolev embedding $W^{1,p}(\Omega_j)\hookrightarrow \mathcal{C}(\overline{\Omega}_j)$ assures that $\widehat{f}$ is again $\mathcal{C}^0$-close to $f_j$ on $\Omega_j$.  

As pointed out, the above construction is 'local'. That is, in order to apply it on a manifold, $f_1(\Omega_1)$ and $f_2(\Omega_2)$ have to lie in the same chart. Therefore, it is natural to seek its analogue valid in the case when these images are contained in two overlapping chart neighborhoods. In this way, we obtain a so-called non-linear version of the Cousin problem which, in the complex case, was solved by Rosay (see e.g. \cite[Lemma 4.5]{CHAK} or \cite[Proposition 1']{Rosay}). We provide an analogue of his results in the case when $f_1$ is fixed.

\begin{theorem}\label{dva} Let $(\Omega_1,\Omega_2)$ be a good pair and let the sets $U_1,U_2\subset \R^{2n}$ be open. Given a diffeomorphism $\Psi\colon U_2\to U_1$ let $J_1\in\mathcal{J}(U_1)$ and $J_2\in\mathcal{J}(U_2)$ be such that $J_2=d\Psi_*(J_1)$. Let $f_1\in\mathcal{O}_{J_1}(\Omega_1)$ be an embedding satisfying $f_1(\overline{\Omega_1\cap\Omega_2})\subset U_1$ and let  $\mathcal{W}_2\subset \mathcal{O}_{J_2}(\Omega_2)$ be a $W^{1,p}(\Omega_2)$-precompact family of maps satisfying $f_2(\overline{\Omega_1\cap\Omega_2})\subset U_2$. For every $\epsilon>0$ there is $\delta>0$ such that for every $f_2\in\mathcal{W}_2$ satisfying $\left\|f_1-\Psi(f_2)\right\|_{W^{1,p}(\Omega_1\cap\Omega_2)}<\delta$ there exist maps $\widehat{f}_1\in \mathcal{O}_{J_1}(\Omega_1)$ and $\widehat{f}_2\in \mathcal{O}_{J_2}(\Omega_2)$ satisfying: $\Psi(\widehat{f}_2)=\widehat{f}_1$ on $\Omega_1\cap\Omega_2$,  
$$\left\|\widehat{f}_1-f_1\right\|_{W^{1,p}(\Omega_1)}<\epsilon\;\;\textrm{and}\;\; \left\|\widehat{f}_2-f_2\right\|_{W^{1,p}(\Omega_2)}<\epsilon.$$
\end{theorem}

Let $(M,J)$ be a smooth almost complex manifold. By Whitney embedding theorem there exist $N\in\mathbb{N}$ and a smooth embedding $i\colon M\to \mathbb{R}^{N}$. This enables us to endow $M$ with a Riemannian metric. Moreover, we can define the Sobolev space of manifold-valued maps defined on $\Omega\subset\mathbb{C}$:
$$W^{1,p}(\Omega,M)=\left\{f\colon \Omega\to M; \; i\circ f\in W^{k,p}(\Omega)\;\textrm{and}\; f(\Omega)\in M \;\textrm{a.e.} \right\}.$$
When $\Omega\subset \mathbb{C}$ is a bounded domain with $\mathcal{C}^1$-boundary we have the following Sobolev type-estimate: there is $C_{\Omega}>0$ such that 
\begin{equation}\label{nova}\textrm{dist}_{\zeta\in \Omega}\left(f(\zeta),g(\zeta)\right)\leq C_{\Omega} \left\| f- g\right\|_{W^{1,p}(\Omega,M)}.\end{equation}  
We denote by $\mathcal{O}_J(\Omega,M)\subset W^{1,p}(\Omega,M)$ the subset of $J$-holomorphic maps.     

\begin{theorem}\label{nov} Let $(\Omega_1,\Omega_2)$ be a good pair and let $(M,J)$ be a smooth almost complex manifold. Let $f_1\in\mathcal{O}_{J}(\Omega_1,M)$ and let $\mathcal{W}_2\subset \mathcal{O}_{J}(\Omega_2,M)$ be a family of maps that is precompact in $W^{1,p}(\Omega_2,M)$. For every $\epsilon>0$ there is $\delta>0$ such that for every $f_2\in\mathcal{W}_2$ satisfying $\left\|i\circ f_1-i \circ f_2\right\|_{W^{1,p}(\Omega_1\cap\Omega_2)}<\delta$ there exists a map $\widehat{f}\in \mathcal{O}_{J}(\Omega_1\cup\Omega_2,M)$ such that: $$\textrm{dist}_{\zeta\in\Omega_1}\left(\widehat{f}(\zeta),f_1(\zeta)\right)<\epsilon\;\;\textrm{ and }\;\; \textrm{dist}_{\zeta\in\Omega_2}\left(\widehat{f}(\zeta),f_2(\zeta)\right)<\epsilon.$$
\end{theorem}
\noindent Theorems 2 and 3 can be understood as a gluing technique similar to those developed by McDuff and Salamon for compact $J$-holomorphic curves  \cite{McDuff2,McDuff}. That is, following their work we provide analogous methods for $J$-holomorphic discs. In particular, we provide a generalized version of Rosay's solution to the Cousin problem in which the Cartan lemma is replaced by these techniques. Such an approach is new even in the integrable case.  

\section{The local Cousin problem} 
In this section, we first introduce the non-linear techniques developed in \cite{BERKUZ} and then prove Theorem \ref{ena}. The reader should note that in the original reference the $\bar{\partial}_J$-equation was discussed only on the unit disc. However, the proofs remain the same for any bounded simply connected domain $\Omega\subset\C$. Therefore, we will mainly omit them. The key tool of our non-linear analysis is the following version of implicit function theorem for Banach spaces. It can be understood as a sufficient condition for the convergence of a Newton-type iteration $x_{n+1}=x_{n}-Q_{x_0}\mathcal{F}(x_n)$, see \cite[Appendix A.3.]{McDuff}. 

\begin{theorem}\label{tri}
\label{theoimpl}
Let $X$ and $Y$ be two Banach spaces and let $\mathcal{F}\colon U \to Y$ be a $\mathcal{C}^1$-map defined on an open set $U\subset X$. Given $x_0\in U$ assume that the differential $d_{x_0}\mathcal{F}$ admits a bounded right inverse denoted by $Q_{x_0}$. Fix $\rho>0$, $L>0$ and $C>0$ such that:
\begin{itemize} 
\item[i)] The operator norm of the right inverse satisfies $\left\|Q_{x_0}\right\|_{op}\leq C$.
\item[ii)] If $\left\|x-x_0\right\|_X<\rho$ then $x\in U$ and 
$\left\|d_x\mathcal{F}-d_{x_0}\mathcal{F}\right\|_{op}\leq L \left\|x-x_0\right\|_X.$
\end{itemize}
The following conclusion is valid: if $\displaystyle \left\|\mathcal{F}(x_0)\right\|_Y<\min\left\{\frac{\rho}{4C},\frac{1}{8C^2L}\right\}$ there exists $x\in U$ such that $\mathcal{F}(x)=0$ and 
$\left\|x-x_0\right\|_X\leq2C\left\|\mathcal{F}(x_0)\right\|_Y.$
\end{theorem}
We apply this statement on a $\bar{\partial}_J$-operator based on (\ref{A}). Precisely, for $p>2$ and a bounded domain $\Omega\subset \C$ let $\mathcal{F}\colon W^{1,p}(\Omega)\to L^p(\Omega)$ be given by
\begin{eqnarray}\label{F}\mathcal{F}(f)=f_{\bar{\zeta}}+A(f)\overline{f_\zeta}.\end{eqnarray}
For $\varphi\in W^{1,p}(\Omega)$ its linearization $d_\varphi\mathcal{F}\colon W^{1,p}(\Omega)\to L^p(\Omega)$ is of the form
$$d_\varphi\mathcal{F}(V) =  V+A(\varphi)\overline{V_{\zeta}}+\sum_{j=1}^n\left(\frac{\partial A}{\partial z_j}(\varphi)V_j+\frac{\partial A}{\partial \bar{z}_j}(\varphi)\overline{V}_j \right) \ \overline{\varphi_\zeta}.$$
We can rewrite this into
$$d_\varphi\mathcal{F}(V) =  V+A(\varphi)\overline{V_{\zeta}}+B_1^\varphi V+B_2^\varphi \overline{V}$$
where $B_1^\varphi$ and $B_2^\varphi$ are matrix functions with coefficients of class $L^p(\Omega)$ and depending continuously on the map $\varphi\in W^{1,p}(\Omega)$. Moreover, in the context of Theorem \ref{tri} we have the following proposition. 

\begin{proposition}\label{stiri} The operator $d_\varphi\mathcal{F}$ satisfies the following two conditions:
\begin{itemize}
\item[i)] The linear map $d_\varphi\mathcal{F}\colon W^{1,p}(\Omega)\to L^p(\Omega)$ admits a bounded right inverse $Q_\varphi$ for every $\varphi\in W^{1,p}(\Omega)$.
\item[ii)] For every $C>0$ there is $L>0$ such that if $\left\|d\varphi\right\|_{L^p(\Omega)}<C$ and  $\left\|\tilde{\varphi}-\varphi\right\|_{W^{1,p}(\Omega)}<1$ we have $\left\|d_{\tilde{\varphi}}\mathcal{F}-d_\varphi\mathcal{F}\right\|_{op}<L \left\|\tilde{\varphi}-\varphi\right\|_{W^{1,p}(\Omega)}.$
\end{itemize}
\end{proposition}
\noindent The proof of the local Lipshitz property in $ii)$ is rather straightforward, see \cite[p. 7]{BERKUZ}. In contrast, $i)$ is provided by \cite[Theorem 2]{BERKUZ} and requires a carefull analysis with singular integral operators. In particular, for $J\equiv J_{st}$ the operator $Q_{\varphi}$ equals the Cauchy-Green operator defined in (\ref{CG}) and is regarded as a bounded map $T_\Omega\colon L^{p}(\Omega)\to W^{1,p}(\Omega)$. However, in general $Q_\varphi$ depends on $\varphi\in W^{1,p}(\Omega)$ and, by the construction, it varies continuously only in the case of the so-called regular complex structures, see \cite[p.\;5]{BERKUZ}. Nevertheless, we can prove the following proposition which is new and will allow us to work with non-regular structures as well.
\begin{proposition}\label{pet}
Let $\mathcal{W}\subset W^{1,p}(\Omega)$ be precompact. There is $C>0$ such that for $\varphi\in \mathcal{W}$ the operator $d_\varphi\mathcal{F}$ admits a right inverse $Q_\varphi$ with $\left\|Q_{\varphi}\right\|_{op}<C$. 
\end{proposition} 
\begin{proof}
We seek $C>0$ such that given $\varphi \in\mathcal{W}$ and $W\in L^{p}(\Omega)$ there is $V\in W^{1,p}(\Omega)$ that satisfies $d_\varphi\mathcal{F}(V)= W$ and $\left\|V\right\|_{W^{1,p}(\Omega)}\leq C \left\|W\right\|_{L^p(\Omega)}.$ As stated in the Proposition  5, for every $\varphi\in\mathcal{W}$ the operator $d_\varphi\mathcal{F}$ admits a bounded right inverse $Q_\varphi$. However, such an inverse is not unique and can be chosen to vary continuously near $\varphi$. 

Since the $L^p$-norm of $d\varphi$ is uniformly bounded for $\varphi\in\mathcal{W}$, the operator $\varphi\mapsto d_\varphi\mathcal{F}$ is locally Lipshitz by part $ii)$ of Proposition \ref{stiri}. Therefore there is $\delta_\varphi>0$ such that if $\left\|\varphi-\tilde{\varphi}\right\|_{W^{1,p}(\Omega)}<\delta_\varphi$ we have
$$\left\|d_{\tilde{\varphi}}\mathcal{F}\circ Q_\varphi-Id\right\|_{op}<\frac{1}{2}.$$
Thus an alternative selection for the right inverse of $d_{\tilde{\varphi}}\mathcal{F}$ is the operator
$$Q_{\tilde{\varphi}}=Q_\varphi\left(d_{\tilde{\varphi}}\mathcal{F}\circ Q_\varphi\right)^{-1}.$$
Moreover, note that such an operator satisfies
$$\left\|Q_{\tilde{\varphi}}\right\|_{op}\leq 2\left\|Q_\varphi\right\|_{op}.$$
Since $\mathcal{W}$ is precompact we can cover it with a finite union of $W^{1,p}$-balls
$$\mathcal{W}\subset \cup_{k=1}^m \mathbb{B}(\varphi_k,\delta_{\varphi_k}).$$ 
Therefore we can produce solutions of the linearized equation obeying the required bound for
$$C=2\cdot \max_{1\leq k\leq m}\left\|Q_{\varphi_k}\right\|_{op}.$$
This produces the family of right inverses $Q_{\varphi}$, $\varphi\in\mathcal{W}$, that we seek.
\end{proof}

The last statement that we need is \cite[Theorem 5]{BERKUZ}. We cite it below. Note that it is a direct implication of Theorem \ref{tri} and Proposition \ref{stiri}. Given a map $\varphi$ with a small $\bar{\partial}_J$-derivative it provides its $J$-holomorphic approximation. 
\begin{theorem}\label{sest}
Given $C>0$ there exists $\delta_C>0$ such that for every map $\varphi \in W^{1,p}(\Omega)$ satisfying 
$$\|d\varphi\|_{L^p(\Omega)}\leq C,\;\; \|Q_\varphi\|_{op}\leq C,\;\;  \left\|\mathcal{F}(\varphi)\right\|_{L^p(\Omega)}<\delta_C,$$
there exists a $J$-holomorphic disc $\widehat{f}\in W^{1,p}(\Omega)$  such that 
$$\left\|\widehat{f}-\varphi\right\|_{W^{1,p}(\Omega)}\leq 2C\left\|\mathcal{F}(\varphi)\right\|_{L^p(\Omega)}.$$ 
\end{theorem}

We are now ready to solve the local version of the Cousin problem. 
\begin{proof}[Proof of Theorem 1]
As in the book of McDuff and Salamon \cite{McDuff2,McDuff} our gluing method consists of two steps. Firstly, we define a family of so-called pregluing maps $\varphi$ defined on $\Omega_1\cup\Omega_2$ and with uniformly $L^p$-small derivatives $\bar{\partial}_J\varphi$. Secondly, applying Theorem \ref{sest}, we provide a $J$-holomorphic approximation $\widehat{f}$ for each of its elements. This is similar to (\ref{osnovni}), where the map $\varphi=\chi f_1+(1-\chi)f_2$ was approximated by 
$$\widehat{f}=\chi f_1+ (1-\chi)f_2 - T_{\Omega_1\cup\Omega_2}\left[\chi_{\bar{\zeta}}(f_1-f_2)\right]=\varphi-T_{\Omega_1\cup\Omega_2}\left[\bar{\partial}\varphi_{J_{st}}\right].$$ 
In the above expression the smallness of $\bar{\partial}_{J_{st}}\varphi$ is guaranteed by the $\mathcal{C}^0$-proximity of the images $f_1(\Omega_1\cap \Omega_2)$ and $f_2(\Omega_1\cap \Omega_2)$. Moreover, in the spirit of Theorem \ref{tri}, the Newton-type iteration ends after the first step. In contrast, we expect the Newton iteration to be infinite in our non-integrable version of this construction while the $L^p$-norm of $\bar{\partial}_J$ will be controlled by 
$$\left\|f_1-f_2\right\|_{W^{1,p}(\Omega_1\cap \Omega_2)}<\delta.$$
That is, we seek an explicit bound for $\delta>0$ in Theorem 1.
\begin{lemma}\label{sedem}
Under the assumptions of Theorem \ref{ena} the set of all possible pregluing maps
$\mathcal{W}=\left\{\chi f_1+(1-\chi)f_2;\; f_1\in\mathcal{W}_1,\;f_2\in\mathcal{W}_2\right\}$
satisfies the following two properties:
\begin{itemize}
\item[i)] There is $C_0>0$ so that $\varphi\in\mathcal{W}$ implies $\left\|\mathcal{F}(\varphi)\right\|_{L^p(\Omega_1\cup\Omega_2)}<C_0\cdot\delta.$
\item[ii)] The set $\mathcal{W}\subset W^{1,p}(\Omega_1\cup \Omega_2)$ is precompact.
\end{itemize}
\end{lemma}
\begin{proof}
Given $\varphi\in\mathcal{W}$ the map $\mathcal{F}(\varphi)$ vanishes everywhere except on the intersection $\Omega_1\cap \Omega_2$. Moreover, on that set we have $\mathcal{F}(\varphi)=I+II$ where
$$I=\chi_{\bar{\zeta}}(f_1-f_2)+A(\varphi)\overline{\chi_\zeta(f_1-f_2)},$$
$$II=\chi\left((f_1)_{\bar{\zeta}}+A(\varphi)\overline{(f_1)_\zeta}\right)+(1-\chi)\left((f_2)_{\bar{\zeta}}+A(\varphi)\overline{(f_2)_\zeta}\right).$$
Since $\mathcal{W}_1$ and $\mathcal{W}_2$ are $W^{1,p}$-bounded, $\chi$ is fixed and $p>2$, the $L^\infty$-norm of all possible pregluing maps $\varphi\in\mathcal{W}$ is uniformly bounded. This implies that on the set $\Omega_1\cap\Omega_2$ we have
$$\left\|I\right\|_{L^p}\leq\left(\left\|\chi_{\bar{\zeta}}\right\|_{L^\infty}+\left\|A(\varphi)\right\|_{L^\infty}\left\|\chi_{\zeta}\right\|_{L^\infty}\right)\delta=C_1\delta,$$
where the constant $C_1>0$ depends on the structure $J$, the pair $(\Omega_1,\Omega_2)$ and the sets $\mathcal{W}_1$ and $\mathcal{W}_2$.
Similarly, we have
$$\left\|II\right\|_{L^p}\leq \left\|A(\varphi)-A(f_1)\right\|_{L^\infty}\left\|(f_1)_\zeta\right\|_{L^p}+\left\|A(\varphi)-A(f_2)\right\|_{L^\infty}\left\|(f_2)_\zeta\right\|_{L^p}<C_2\delta.$$
Hence, we can set $C_0=C_1+C_2$ to prove $i)$. The point $ii)$ follows simply from the fact that the sets $\mathcal{W}_1$ and $\mathcal{W}_2$ are precompact in the $W^{1,p}$-topology.  
\end{proof}

\noindent After defining all possible pregluing maps, we proceed to the second step. That is, we show that Theorem \ref{sest} can be (uniformly) applied on $\mathcal{W}$. Note that there is $C>0$ such that for every $\varphi\in\mathcal{W}$ we have 
$$\left\|d\varphi\right\|_{L^{p}(\Omega_1\cup\Omega_2)}<C\;\;\textrm{and}\;\; \left\|Q_\varphi\right\|_{op}<C.$$ 
The first condition is provided by the fact that $\mathcal{W}$ is precompact, the second one is guaranteed by Proposition \ref{pet}. 
Finally, let $\delta_C>0$ be as in Theorem \ref{sest} and let $C_0>0$ be the constant from part $i)$ of Lemma \ref{sedem}. Assume that
$$\delta<\min\left\{\frac{\delta_C}{C_0}\frac{\epsilon}{4C},\frac{\epsilon}{2\left\|\chi\right\|_{L^\infty} }\right\}.$$ Then by Theorem \ref{sest} every $\varphi\in \mathcal{W}$ admits a $J$-holomorphic approximation $\widehat{f}$. Moreover, these are the maps we seek since for $j=1,2$, we have
$$\left\|\widehat{f}-f_j\right\|_{W^{1,p}(\Omega_j)}\leq \left\|\widehat{f}-\varphi\right\|_{W^{1,p}(\Omega_j)}+\left\|\varphi-f_j\right\|_{W^{1,p}(\Omega_j)}<\epsilon.$$
This completes the proof.
\end{proof}


\section{The non-linear problem}        
This section is devoted to the proof of Theorem \ref{dva}. Again, we rely strongly on \cite{BERKUZ}, but due to the nature of the problem we have to work in two local charts at the same time. In particular, we will apply our Newton-type iteration to the map $\mathcal{F} \colon W^{1,p}(\Omega_1)\times W^{1,p}(\Omega_2)\to L^{p}(\Omega_1)\times L^p(\Omega_2),$ given by 
$$\mathcal{F}(\varphi_1,\varphi_2)=(\mathcal{F}_1(\varphi_1),\mathcal{F}_2(\varphi_2)),$$
where $\mathcal{F}_1$ and $\mathcal{F}_2$ are the $J$-holomorphicity operators defined as in (\ref{A}) for $J_1$ and $J_2$ (we take the $L^1$-norm on the direct sum of Banach spaces). Our goal will be to find an appropriate right inverse for its linearization so that the produced limit will be an element of the set  
$$\mathcal{M}=\left\{(\varphi_1,\varphi_2)\in W^{1,p}(\Omega_1)\times W^{1,p}(\Omega_2);\;\; \Psi(\varphi_2)=\varphi_1\;\;\textrm{on}\;\; \Omega_1\cap\Omega_2 \right\}.$$
As in the local case, we set
$$\left\|f_1-\Psi(f_2)\right\|_{W^{1,p}(\Omega_1\cap\Omega_2)}<\delta\leq 1 $$
and seek an appropriate bound for $\delta>0$. However, this time we do not state it explicity. We only show that Theorem \ref{dva} holds if $\delta$ is 'small enough'.
 
Let $f_2\in\mathcal{W}_2$ be arbitrary. Recall that $\chi$ is a cut-off function corresponding to the pair $(\Omega_1,\Omega_2)$. For small $\delta$ we can define $\varphi_1\colon\Omega_1\to \R^{2n}$ given by   
\begin{eqnarray}\label{fi1}\varphi_1=\chi f_1+(1-\chi)\Psi(f_2).\end{eqnarray}
Furthermore, we can set $\varphi_2=\Psi^{-1}(\varphi_1)$ on $\Omega_1\cap\Omega_2$ and then extend this map to the whole $\Omega_2$ by taking $\varphi_2=f_2$ on $\Omega_2\setminus \Omega_1$. Note that $(\varphi_1,\varphi_2)\in\mathcal{M}.$ 
Moreover, it follows directly from part $i)$ in Lemma \ref{sedem} that there is $C_0>0$ (depending also on the diffeomorphism $\Psi$) such that
$$\left\|\mathcal{F}(\varphi_1,\varphi_2)\right\|_{W^{1,p}(\Omega_1)\times W^{1,p}(\Omega_2)}<C_0\cdot\delta.$$
Thus, the pair $(\varphi_1,\varphi_2)$ is a natural candidate for the pregluing map in our construction. Furthermore, since $f_1$ is fixed and $f_2\in\mathcal{W}_2$ lies in a precompact family, the $W^{1,p}$-norm of $\varphi_1$ and $\varphi_2$ can be uniformly bounded. Therefore, a statement similar to part $ii)$ in Proposition \ref{stiri} can be proved.  

Hence, the only missing step in our gluing construction is to find a family of appropriate and uniformly bounded right inverses for $d_{(\varphi_1,\varphi_2)}\mathcal{F}$. Indeed, after that a statement similar to Theorem \ref{sest} can be proved and the desired holomorphic pair $(\widehat{f}_1,\widehat{f}_2)$ can be provided. We leave these details to the reader and focus solely on the construction of appropriate inverses. It turns out that, in order to find them, we have to slightly change the idea from (\ref{osnovni}). Therefore, we begin by explaining our new approach in this simplest case.   
\vskip 0.2 cm
\noindent \textbf{New approach for $\Psi=id$ and $J\equiv J_{st}$}.
\vskip 0.1 cm
\noindent In the integrable case $\bar{\partial}_{J_{st}}$ is linear and on $\Omega_1\cup\Omega_2$ its right inverse can be chosen to be $T_{\Omega_1\cup\Omega_2}$. Hence, at the first sight, a natural candidate for the inverse of $d_{(\varphi_1,\varphi_2)}\mathcal{F}=(d_{\varphi_1}\mathcal{F}_1,d_{\varphi_2}\mathcal{F}_2)$ seems to be $(T_{\Omega_1},T_{\Omega_2})$. However, this is not the case since, in general, we have $T_{\Omega_1}(V)\neq T_{\Omega_2}(V)$ on $\Omega_1\cap\Omega_2$. That is, such an operator does not provide variations lying in the tangent space $T_{(\varphi_1,\varphi_2)}\mathcal{M}$. Hence, we seek its replacement that can be computed with values from $\Omega_1$ and $\Omega_2$ separately and still maps into $W^{1,p}(\Omega_1\cup\Omega_2)$. 

We define the following two compact sets
\begin{eqnarray}\label{kompakta}K_1=\overline{\Omega_1\cap\Omega_2}\cap\left\{\chi\leq\frac{1}{3}\right\},\;\; K_2=\overline{\Omega_1\cap\Omega_2}\cap\left\{\chi\geq\frac{2}{3}\right\}.\end{eqnarray}
Given $\gamma>0$ let $P_\gamma$ be a complex polynomial such that 
$$\left\|P_\gamma\right\|_{L^{\infty}(K_1)}<\gamma\;\;\textrm{and}\;\; \left\|P_\gamma-1\right\|_{L^{\infty}(K_2)}<\gamma.$$
Let $\beta_1$ and $\beta_2$ be two smooth cut-off functions defined on $\Omega_1\cap\Omega_2$ and and having the following properties:
$\beta_1(\zeta)\equiv 0$ where $\chi(\zeta)<\frac{1}{9}$; $\beta_1(\zeta)\equiv 1$ where $\chi(\zeta)>\frac{2}{9}$; $\beta_2(\zeta)\equiv 0$ where $\chi(\zeta)<\frac{7}{9}$; $\beta_2(\zeta)\equiv 0$ where $\chi(\zeta)>\frac{8}{9}$. 

Given $V\in L^{p}(\Omega_1\cup \Omega_2)$ we denote by $V_1$ and $V_2$ its restrictions to $\Omega_1$ and $\Omega_2$, respectively. Furthermore, in what is written below we treat $Q_j=T_{\Omega_j}$, $j=1,2$, as a bounded map from $L^p(\Omega_j)$ to $W^{1,p}(\Omega_j)$ although the classical Cauchy-Green operator extends continuously to the whole $\Omega_1\cup\Omega_2$. Note that the map 
$$W(V)=Q_2(V_2)-Q_1(V_1)$$
is well defined and satisfies $\bar{\partial}_{J_{st}}W(V)=0$ on $\Omega_1\cap\Omega_2$.   
Hence, we can define the operator $\widehat{Q}\colon L^{p}(\Omega)\to W^{1,p}(\Omega)$ given by
$$\widehat{Q}(V)=(1-\beta_2)\left(Q_1(V_1)+\beta_1 P_\gamma W(V)\right)+\beta_2Q_2(V_2).$$
A straightforward computation shows that
$$\bar{\partial}_{J_{st}}\widehat{Q}(V)=V+(\beta_1)_{\bar{\zeta}}P_\gamma W(V)+(\beta_2)_{\bar{\zeta}}(1-P_{\gamma})W(V).$$
Hence,
$$\left\|\bar{\partial}_{J_{st}}\widehat{Q}(V)-V\right\|_{L^p(\Omega_1\cup\Omega_2)}\leq C\cdot\gamma \left\|V\right\|_{L^p(\Omega_1\cup\Omega_2)}$$
where $C>0$ depends on the operator bounds for $T_{\Omega_j}$ and the derivatives of $\beta_1$ and $\beta_2$. Thus, provided that $\gamma\leq \frac{1}{2C}$ a new right inverse for $\bar{\partial}_{J_{st}}$ on $\Omega_1\cup\Omega_2$ can be defined by
$Q=\widehat{Q}\left(\bar{\partial}_{J_{st}}\circ\widehat{Q}\right)^{-1}.$ Moreover, note that its norm is bounded by twice the norm of $\widehat{Q}$. That is, the $W^{1,p}$-size of $P_\gamma$, $\beta_1$, $\beta_2$, and the operator norms of $Q_1$ and $Q_2$. \qed

\vskip 0.2 cm
In what follows we would like to extend this construction to the case of non-integrable stuctures. Recall that
$$d_{\varphi_1}\mathcal{F}_1(V)=V_{\bar{\zeta}}+A(\varphi_1)\overline{V_\zeta}+B_1^{\varphi_1} V+B_2^{\varphi_1} \overline{V}.$$ 
However, since $f_1$ is a fixed embedding, we can assume that, after a change of coordinates, $J_1(f_1)=J_{st}$ and $A_1(f_1)=0$, see e.g. \cite[Appendix A.2.]{IR}. That is, for all possible $\varphi_1$ the matrix function $A(\varphi_1)$ can be assumed to be $\delta$-close to the zero matrix. Therefore, it is enough to work with the so-called generalized analytic vectors $\mathcal{O}_{B_1,B_2}(\Omega)\subset W^{1,p}(\Omega)$ satisfying the equation
$$\bar{\partial}_{B_1,B_2}V= V_{\bar{\zeta}}+B_1 V+B_2 \overline{V}=0,$$
where $B_1$ and $B_2$ are matrix functions with coefficients of class $L^p(\Omega)$. 
Precisely, in order to replace the complex polynomial $P_\gamma$ used in the new approach above, we apply the following Runge-type theorem which is due to Goldschmidt \cite[Theorem 3.1]{Goldschmidt}.

\begin{theorem}\label{osem} Let $\Omega\subset \C$ be bounded and let $B_1$, $B_2$ be matrix functions with coefficients in $L^p(\Omega)$. Let $\Omega_0\subset\Omega$ be a domain whose complement admits no relatively compact connected components. Given $V_0\in\mathcal{O}_{B_1,B_2}(\Omega_0)$, $\epsilon>0$ and compact set $K\subset \Omega_0$ there is $V\in\mathcal{O}_{B_1,B_2}(\Omega)$ such that
$\left\|V_0-V\right\|_{L^\infty(K)}<\epsilon.$
\end{theorem}
\noindent\textbf{Remark.} The reader should note that in the original reference, this theorem is stated in terms of weak solutions of $\overline{\partial}_{B_1,B_2}V=0$ belonging to $L_{loc}^q(\Omega)$ and $L_{loc}^q(\Omega_0)$, where $\frac{1}{p}+\frac{1}{q}=1$. However, the present version can be established via standard bootstrapping arguments (check e.g. \cite[Section III/3]{VEKUA}). 
\vskip 0.1 cm
We use this statement in order to prove the following crucial lemma.
\begin{lemma}\label{devet}
Let $\gamma>0$ and let $\varphi_1$, $K_1$, $K_2$ be defined as in (\ref{fi1}) and (\ref{kompakta}). There is $C_\gamma>0$ such that for every $W\in W^{1,p}(\Omega_1\cap\Omega_2)$ satisfying $d_{\varphi_1}\mathcal{F}_1(W)=0$ there exists a matrix function $P_\gamma$ satisfying:  
\begin{itemize}
\item[i)] $\left\|P_\gamma\right\|_{L^\infty(K_1)}<\gamma$ and $\left\|P_\gamma-Id\right\|_{L^{\infty}(K_2)}<\gamma.$
\item[ii)] $\left\|d_{\varphi_1}\mathcal{F}_1\left(P_\gamma W\right) \right\|_{W^{1,p}(\Omega_1\cap\Omega_2)}<(\gamma+C_\gamma\delta) \cdot\left\|W\right\|_{W^{1,p}(\Omega_1\cap\Omega_2)}.$ 
\item[iii)] The $W^{1,p}(\Omega_1\cap\Omega_2)$-norm of the coefficients of $P_\gamma$ is bounded by $C_\gamma$. 
\end{itemize}
\end{lemma}
\begin{proof} First note that one can take the same matrix $P_\gamma$ for every $\lambda W$ where $\lambda\in\mathbb{C}$ and $d_{\varphi_1}\mathcal{F}_1(W)=0$. Hence, it suffices to prove this statement only for $W^{1,p}$-unitary solutions of this equation, that is, for $\left\|W\right\|_{W^{1,p}(\Omega_1\cap\Omega_2)}=1.$ 

We begin with the case when $A(\varphi_1)=0$ and $d_{\varphi_1}\mathcal{F}_1(W)=\bar{\partial}_{B_1^{\varphi_1},B_2^{\varphi_1}}W=0$. We would like to find $P_\gamma$ such that $\bar{\partial}_{B_1^{\varphi_1},B_2^{\varphi_1}}\left(P_\gamma W\right)=0$. Equivalently,
$$\left(P_\gamma\right)_{\bar{\zeta}}W+\left(B^{\varphi_1}_1P_\gamma-P_\gamma B^{\varphi_1}_1\right)W+\left(B^{\varphi_1}_2\overline{P_\gamma}-P_\gamma B^{\varphi_1}_2\right)\overline{W}=0.$$
Let $D$ be a diagonal matrix with entries $\frac{\overline{W}_j}{W_j}.$ Then a sufficient condition for the above equation to hold is
$$\left(P_\gamma\right)_{\bar{\zeta}}+\left(B^{\varphi_1}_1P_\gamma-P_\gamma B^{\varphi_1}_1-P_\gamma B^{\varphi_1}_2D\right)+B^{\varphi_1}_2\overline{P_\gamma}D=0.$$ 
However, this matrix equation can be seen as a generalized analytic equation for a $n^2$-dimensional vector $P_\gamma$. Therefore, one can apply the Runge theorem. Indeed, Theorem \ref{osem} allows $L^p$-regular coefficients. Hence $B^{\varphi_1}_1, B^{\varphi_1}_2$ and $D$ can be extended as zero matrix functions to some neighorhood $\Omega$ of $\overline{\Omega_1\cap\Omega_2}$. Moreover, note that $P_\gamma=0$ and $P_\gamma=Id$ solve such an equation on some small neighborhood $\Omega_0$ of $K_1\cup K_2$. Thus there is a generalized analytic solution $P_\gamma$ satisfying $\bar{\partial}_{B_1^{\varphi_1},B_2^{\varphi_1}}\left(P_\gamma W\right)=0$ and the property $i)$.  

Of course, by the construction, such a matrix function $P_\gamma$ depends on a concrete solution of $\bar{\partial}_{B_1^{\varphi_1},B_2^{\varphi_1}}W=0$. Therefore, we can not yet expect to have a uniform bound needed in $iii)$. However, for $ii)$ we do not need the vanishing of $\bar{\partial}_{B^{\varphi_1}_1,B^{\varphi_1}_2}(P_\gamma W)$, we just seek an appropriate $W^{1,p}$-bound for this expression. Hence, a locally constant selection of $P_\gamma$ can be made. Indeed, given $W\in\mathcal{O}_{B_1^{\varphi_1},B_2^{\varphi_1}}(\Omega_1\cap\Omega_2)$ the generalized Cauchy integral formula yields the following representation on every bounded domain $\mathcal{D}\subseteq\Omega_1\cap\Omega_2$ with a $\mathcal{C}^1$ boundary:
$$W(\zeta)=-T_{\mathcal{D}}\left(B^{\varphi_1}_1W+B^{\varphi_1}_2\overline{W}\right)(\zeta)+\frac{1}{2\pi i}\int_{\partial \mathcal{D}}\frac{f(z)}{z-\zeta}d\zeta.$$
Hence, provided that $\left\|W\right\|_{W^{1,p}(\Omega_1\cap\Omega_2)}=1$ the first sumand is uniformly $W^{2,p}$-bounded. Moreover, by Sobolev embedding theorem all such solutions admit a uniform $\mathcal{C}^0$-bound on $\partial \mathcal{D}$. This together with the classical Montel argument implies that the subset of generalized analytic vectors admiting a unitary $W^{1,p}$-norm is precompact in the space $W^{1,p}_{loc}(\Omega_1\cap\Omega_2)$. 

Let now $W$ be such a $W^{1,p}$-unitary solution and $P_\gamma$ the corresponding matrix function for which  $\bar{\partial}_{B_1^{\varphi_1},B_2^{\varphi_1}}\left(P_\gamma W\right)=0$. Then for any other $W^{1,p}$-unitary solution $\widehat{W}$ we have
$$\bar{\partial}_{B_1^{\varphi_1},B_2^{\varphi_1}}\left(P_\gamma \widehat{W}\right)=P_\gamma\left(\widehat{W}-W\right)_{\bar{\zeta}}+\left(\left(P_\gamma\right)_{\bar{\zeta}}+B_1P_{\gamma}+B_2\overline{P_{\gamma}}\right)\left(\widehat{W}-W\right).$$
Hence, there exists a compact set $K_W\subset\Omega_1\cap\Omega_2$ large enough so that
$$\left\|\bar{\partial}_{B_1^{\varphi_1},B_2^{\varphi_1}}\left(P_\gamma \widehat{W}\right)\right\|_{L^p(\Omega_1\cap\Omega_2\setminus K_W)}<\frac{\gamma}{2}.$$
Moreover, there is $\delta_W>0$ such that $\left\|W-\widehat{W}\right\|_{W^{1,p}_{loc}(\Omega_1\cap\Omega_2)}<\delta_W$ implies
$$\left\|\bar{\partial}_{B_1^{\varphi_1},B_2^{\varphi_1}}\left(P_\gamma \widehat{W}\right)\right\|_{L^p(K_W)}<\frac{\gamma}{2}.$$
Hence, a $W^{1,p}$-unitary solution $\widehat{W}$ that is $\delta_W$-close to $W$ in $W^{1,p}_{loc}(\Omega_1\cap\Omega_2)$ satisfies $ii)$ with $\delta=0$ if the matrix $P_\gamma$ satisfies $\bar{\partial}_{B_1^{\varphi_1},B_2^{\varphi_1}}\left(P_\gamma W\right)=0 $. This together with the above explained precompactness of such solutions in $W^{1,p}_{loc}(\Omega_1\cap\Omega_2)$ gives the desired conclusion for $A(\varphi_1)=0$.

Recall now that $\varphi_1=\chi f_1+(1-\chi)\Psi(f_2)$ where $A(f_1)$ was assumed to vanish and $f_2$ belongs to a precompact $W^{1,p}$-subset. Therefore, for all possible $\varphi_1$ we can assume that for some $C_1>0$ we have 
$$\left\|A(\varphi_1)\right\|_{L^\infty(\Omega_1)}<C_1\cdot \delta.$$ 
Further, let $d_{\varphi_1}\mathcal{F}_1(W)=0$ and $\left\|W\right\|_{W^{1,p}(\Omega_1\cap\Omega_2)}=1$. By \cite[Corollary 3.6]{ST} there is $S\in W^{1,p}(\Omega_1\cap\Omega_2)$ such that
$$S_{\bar{\zeta}}+B^{\varphi_1}_1S+B^{\varphi_1}_2\overline{S}=A(\varphi_1)\overline{W_{\zeta}}.$$
Moreover, provided that $\varphi_1$ belongs to a precompact set, we can use the same trick as in the proof of Proposition \ref{pet} in order to provide a uniform bound for the operator norms for right inverses of all possible $\bar{\partial}_{B_1^{\varphi_1},B_2^{\varphi_1}}$-operators. That is, there is $C_2>0$ such that
$$\left\|S\right\|_{W^{1,p}(\Omega_1\cap\Omega_2)}<C_2\cdot \delta.$$
Finally, note that $\bar{\partial}_{B_1^{\varphi_1},B_2^{\varphi_1}}(W+S)=0$. Hence, by what we have proved above, there exists a matrix function $P_\gamma$ such that 
$$\left\|\bar{\partial}_{B_1^{\varphi_1},B_2^{\varphi_1}}\left(P_\gamma(W+S)\right)\right\|_{L^p(\Omega_1\cap\Omega_2)}\leq \gamma \cdot \left\|W+S\right\|_{W^{1,p}(\Omega_1\cap\Omega_2)}\leq \gamma \cdot (1+C_2\cdot \delta).$$ 
Furthermore, we have
$$d_{\varphi_1}\mathcal{F}_1(P_\gamma W)= \bar{\partial}_{B_1^{\varphi_1},B_2^{\varphi_1}}\left(P_\gamma(W+S)\right)-\bar{\partial}_{B_1^{\varphi_1},B_2^{\varphi_1}}\left(P_\gamma S\right)+A(\varphi_1)\overline{\left(P_\gamma W\right)_{\zeta}}.$$
Recall that the $W^{1,p}$-norm of $P_\gamma$ was proved to be uniformly bounded in the generalized analytic case. Hence, this combined with the fact that the $L^p$-norms of $B^{\varphi_{1}}_1$ and $B^{\varphi_{1}}_2$ are uniformly bounded for all possible $\varphi_1$ gives the desired conclusion. That is, $iii)$ holds for $W$ if we choose $P_\gamma$ to be the matrix corresponding to the generalized analytic vector $W+S$. 
\end{proof}
We can now turn towards the construction of the right inverse. Note that
$$T_{(\varphi_1,\varphi_2)}\mathcal{M}=\left\{(V_1,V_2)\in W^{1,p}(\Omega_1)\times W^{1,p}(\Omega_2);\;\; d_{\varphi_2}\Psi(V_2)=V_1\;\;\textrm{on}\;\; \Omega_1\cap\Omega_2 \right\}$$
and that the linearization of $\mathcal{F}$ at $(\varphi_1,\varphi_2)\in\mathcal{M}$ is given by
$$d_{(\varphi_1,\varphi_2)}\mathcal{F}(V_1,V_2)=\left(d_{\varphi_1}\mathcal{F}_1(V_1),d_{\varphi_2}\mathcal{F}_2(V_2)\right).$$
Given an arbitrary pair $(S_1,S_2)$ of $L^p$-variations along $\mathcal{F}(\varphi_1,\varphi_2)$, we can invert each component with $V_1=Q_{\varphi_1}(S_1)$ and $V_2=Q_{\varphi_2}(S_2)$ where the operators $Q_{\varphi_1}$ and $Q_{\varphi_2}$ are the inverses of $d_{\varphi_1}\mathcal{F}_1$ and $d_{\varphi_2}\mathcal{F}_2$ provided by Propositions \ref{stiri} and \ref{pet}. However, in general, $(V_1,V_2)\notin T_{(\varphi_1,\varphi_2)}\mathcal{M}$ since  
$$d_{\varphi_2}\Psi(V_2)\neq V_1.$$ 
Still, the pair $(V_1,V_2)$ can be considered as a pair of variations along the maps $\varphi_1$ and $\varphi_2$, respectively. Since $J_2=d\Psi_*(J_1)$ on $\Omega_1\cap\Omega_2$ we have
$$\mathcal{F}_1(\varphi_1)=\mathcal{F}_1(\Psi(\varphi_2)).$$
This means that $V_1$ and $V_2$ satisfy
$$d_{\varphi_1}\mathcal{F}_1(V_1)=d_{\Psi(\varphi_2)}\mathcal{F}_1\circ d_{\varphi_2}\Psi (V_2)=d_{\varphi_1} \mathcal{F}_1 \circ d_{\varphi_2}\Psi(V_2).$$
In particular, $W=d_{\varphi_2}\Psi(V_2)-V_1$ satisfies 
$$d_{\varphi_1}\mathcal{F}_1(W)=W_{\bar{\zeta}}+A_1(\varphi_1)\overline{W_\zeta}+B_1W+B_2\overline{W}=0.$$ 
Therefore, we can apply Lemma \ref{devet}.

Let $K_1$, $K_2$, $\beta_1$ and $\beta_2$ be defined as in (\ref{kompakta}). Given $\gamma>0$ that will be fixed below, let $P_\gamma$ be the matrix corresponding to our $W$ in Lemma \ref{devet}. We can then redefine the inversions $V_1=Q_{\varphi_1}(S_1)$ and $V_2=Q_{\varphi_2}(S_2)$ as follows. First, on $\Omega_1$ we define
$$\tilde{V}_1=(V_1+\beta_1 P_\gamma W),$$
$$\widehat{V}_1=(1-\beta_2)\; \tilde{V}_1 + \beta_2 \;d_{\varphi_2}\Psi(V_2).$$ 
Further, on the intersection $\Omega_1\cap\Omega_2$ we define
$$\widehat{V}_2=d_{\varphi_2}\Psi^{-1}(\widehat{V}_1).$$
and then extend this map to the rest of $\Omega_2$ by setting $\widehat{V}_2=Q_{\varphi_2}(S_2)$ on $\Omega_2\setminus\Omega_1$. Finally, we define $\widehat{Q}(S_1,S_2)=(\widehat{V}_1,\widehat{V}_2)$.

Note that, by the construction, $\widehat{Q}(V_1,V_2)\in T_{(\varphi_1,\varphi_2)}\mathcal{M}$. However, this is only an approximation of a true right inverse for $d_{(\varphi_1,\varphi_2)}\mathcal{F}$. Indeed, firstly because $d_{\varphi_1}\mathcal{F}_1(P_\gamma W)\neq 0$ and secondly since $d_{\varphi_1}\mathcal{F}_1(\widehat{V}_1)$ and $d_{\varphi_2}\mathcal{F}_2(\widehat{V}_2)$ include also terms with derivatives of $\beta_1$ and $\beta_2$. Nevertheless, by Lemma \ref{devet}, the operator norm of the first expression is controlled by $\gamma>0$ and the $\delta$-difference between $f_1$ and $\Psi(f_2)$ on $\Omega_1\cap\Omega_2$. Furthermore, the terms including derivatives of $\beta_1$ and $\beta_2$ are controlled by the $\gamma$-distance between $V_1$ and $\tilde{V}_1$ or $\widehat{V}_1$ and $d_{\varphi_2}\Psi(V_2)$, respectively. Therefore, similarly as in the case with $J_{st}$, one can conclude that there are constants $C>0$ and $C_\gamma>0$ (depending also on the diffeomorphism $\Psi$) such that 
$$\left\|d_{(\varphi_1,\varphi_2)}\mathcal{F}\circ\widehat{Q}-Id\right\|_{op}<C\cdot\gamma+C_\gamma\cdot \delta.$$
We now set $\gamma>0$ so that $C\cdot\gamma<\frac{1}{4}$ and $\delta>0$ so that $C_\gamma\cdot\delta<\frac{1}{4}$. 
The family of uniformly bounded right inverses for $d_{(\varphi_1,\varphi_2)}\mathcal{F}$ can be defined by
$$Q_{(\varphi_1,\varphi_2)}=\widehat{Q}\circ \left(d_{(\varphi_1,\varphi_2)}\mathcal{F}\circ\widehat{Q}\right)^{-1}.$$
Indeed, the operator norm of the elements of this family depends on $P_\gamma$, where $\gamma>0$ is fixed, and the operator norms of $Q_{\varphi_1}$ and $Q_{\varphi_2}$. The later can again be uniformly bounded by Proposition \ref{pet} and part ii) of Lemma \ref{sedem}. 

\section{The gluing technique for discs on manifolds}
We conclude this paper by proving the Theorem 3. Let $\dim_{\mathbb{R}}M=2n$ and let $f\in\mathcal{O}_J(\Omega,M)$ be arbitrary. By \cite[p. 12]{ST} its graph $L_f(\zeta)=(\zeta,f(\zeta))$ admits a coordinate neighborhood 
$\Phi_{f}\colon V_{f}\subset\mathbb{R}^{2}\times M\to U_f\subset\mathbb{R}^{2n+2}$ such that $(d\Phi_{f})_*(J_{st}\otimes J)$ equals the standard structure (of the space $\mathbb{R}^{2n+2}$) along the image of $\Phi_{f}\circ L_f$ and that $\det\left[(d\Phi_{f})_*(J_{st}\otimes J)+J_{st}\right]\neq 0$ on $U_f$. We can apply this fact to $f_1\in\mathcal{O}_J(\Omega_1,M)$ and every $f_2\in\mathcal{W}_2$. In particular, we denote by $\Phi_{f_1}\colon V_{f_1}\to U_{f_1}$ the chart containing the image of $L_{f_1}$.

Let $f_2\in\mathcal{W}_2$. By (\ref{nova}) there is $\delta_{f_2}>0$ such that for every $g_2\colon\Omega_2\to M$ with $\left\|g_2-f_2\right\|_{W^{1,p}(\Omega_2,M)}<\delta_{f_2}$ we have $L_{g_2}(\Omega_2)\subset V_{f_2}$. Moreover, since $\mathcal{W}_2$ is precompact in $W^{1,p}(\Omega_2,M)$ there exist a finite set of coordinate charts $\Phi_{j}\colon V_{j}\subset\mathbb{R}^{2}\times M\to U_j\subset\mathbb{R}^{2n+2}$, $j\in\left\{1,2,\ldots,m\right\}$, such that each graph $L_{f_2}$, $f_2\in\mathcal{W}_2$, belongs to some $V_j$ and that $(d\Phi_{j})_*(J_{st}\otimes J)\in\mathcal{J}(U_j)$.

Let $\tilde{f}_1=\Phi_{f_1}\circ L_{f_1}$, let $\tilde{f}_2^j=\Phi_{j}\circ L_{f_2}$ and let
$$\mathcal{W}_2^j=\left\{\tilde{f}_2^j;\; f_2\in \mathcal{W}_2,\; L_{f_2}(\Omega_2)\subset V_{j} \;\textrm{and}\; L_{f_2}(\Omega_1\cap\Omega_2)\subset V_{f_1}\right\}.$$
Note that the images of $i\circ f_1$ and $i\circ f_2$, $f_2\in\mathcal{W}_2$, are contained in a compact subset of $\mathbb{R}^N$ and that for $\Phi\in \left\{\Phi_{f_1},\Phi_j\right\}$ and $k\in\left\{1,2\right\}$ we have 
$$\Phi\circ L_{f_k}=\left(\Phi\circ L\circ i^{-1}\right)\circ \left(i\circ f_k\right).$$
Hence the sets $\mathcal{W}_2^j$ are precompact in $W^{1,p}(\Omega_2,\mathbb{R}^{2n+2})$. Therefore we can apply Theorem $2$ for a fixed map $\tilde{f}_1\in\mathcal{O}_J(\Omega_1,\mathbb{R}^{2n+2})$, elements $\tilde{f}^j_2$ from $\mathcal{W}_2^j$ and a smooth diffeomorphism $\Psi_j=\Phi_{f_1}\circ\Phi_{j}^{-1}$ mapping from $U_{j}$ to $U_{f_1}$. In particular, we can find $\tilde{\delta}_j>0$ such every  $\tilde{f}_2\in \mathcal{W}_2^j$ satisfying
$$\left\|\tilde{f}_1-\Psi_j(\tilde{f}_2^j)\right\|_{W^{1,p}(\Omega_1\cap\Omega_2)}<\tilde{\delta}_j$$
can be glued with $\tilde{f}_1$ (t.i. we can solve this non-linear Cousin problem). Moreover, this implies that there is $\delta_j>0$ such that if $\tilde{f}_2^j\in \mathcal{W}_2^j$ and if 
$$\left\|i\circ f_1- i \circ f_2\right\|_{W^{1,p}(\Omega_1\cap\Omega_2,\mathbb{R}^N)}<\delta_j$$ 
there exists $\widehat{f}\in\mathcal{O}_J(\Omega_1\cup\Omega_2,M)$ that glues together $f_1$ and $f_2$ in the desired way. Hence, it suffices to set $\delta>0$ to be the minimum of all $\delta_j$'s.

\vskip 0.1 cm
\noindent  {\it Acknowledgments.} The research of the author was supported in part by grants P1-0291, J1-9104, J1-1690 and BI-US/19-21-108 from ARRS, Republic of Slovenia. His work on the non-linear Cousin problem was initialized in Beirut, September 2018. Therefore, he would like to thank F. Bertrand for his hospitality and fruitful discussions. He would also like to thank B. Drinovec-Drnov\v{s}ek and F. Forstneri\v{c} for their useful comments concerning the final version of the paper.

\end{document}